\documentclass[a4paper]{article}
\usepackage{amssymb}
\usepackage{amsthm}
\usepackage{amsmath}
\usepackage{enumerate}
\usepackage[latin2]{inputenc}
\theoremstyle{plain}
\newtheorem{thm}{Theorem}[section]

\newtheorem{cor}[thm]{Corollary}
\newtheorem{lem}[thm]{Lemma}

\newcommand{\FF}[1]{\mathbb F_{#1}}

\newcommand{\sbs}{\subseteq}

\newcommand{\nor}{\vartriangleleft}

\newcommand{\ud}[1]{\underline{#1}}
\newcommand{\dg}{\delta (g)}
\newcommand{\pg}{\pi (g)}
\newcommand{\monomialbasis}{\ud u_1,\ud u_2,\ldots, \ud u_e}
\newcommand{\mb}{\monomialbasis}
\newcommand{\s}{\scriptstyle}
\newcommand{\twolineindex}[2]{\begin{array}{c}{\scriptstyle #1}\\{\scriptstyle #2}\end{array}} 

\begin{document}
\title{On the Orbits of Solvable Linear Groups} 
\author{Zolt\'an Halasi and K\'aroly Podoski} 
\date{} 
\maketitle 
\begin{abstract}\noindent  Let $G$ be a solvable linear group acting on the
finite vectorpace $V$ and assume that $(|G|,|V|)=1$.  In this paper we
find $x,y\in V$ such that $C_G(x)\cap C_G(y)=1$. In particular, this
answers a question of I. M. Isaacs. We complete some results of
S. Dolphi, A. Seress and T. R. Wolf.
\end{abstract}     
\section{Introduction}
One basic concept for computing with permutation groups is the notion
of a base: For a permutation group $G\leq\textrm{Sym }(\Omega)$ a set
$\{\omega_1,\omega_2,\ldots,\omega_n\}\sbs \Omega$ (or rather an
ordered list) is called a base for $G$ if only the identity fixes all
of the elements of this set. There are a number of algorithms for
permutation groups related to the concept of base, and these algorithms
are faster if the size of the base is small. Hence it is useful to
find small bases to permutation groups.
Of course, we cannot except a good result in general, since
taking the natural action of $S_n$, the minimal size of a base is
$n-1$. On the other hand, there are a number of results if 
$G$ is solvable, the action of $G$ is primitive, or $(|G|,|\Omega|)=1$.

A size of a base of a permutation group $G\leq\textrm{Sym }(\Omega)$ 
is at least $\log |G|/\log |\Omega|$. It is a conjecture of L.\ Pyber
\cite{pyber} that for a primitive permutation group $G$ there is a base of size
less than $C\log |G|/\log |\Omega|$ for some universal constant $C$.  
For solvable groups, there is a more general result:
It was proved by A\'.\ Seress \cite{seress} that all primitive 
solvable permutation group has a base of size at most four. According
to the O'Nan--Scott Theorem, any such group is of affine type. 
However, in general there is no universal upper bound on the 
minimal base size of an affine group.

The situation changes, if we consider coprime affine groups.
A permutation group $G\leq\textrm{Sym }\Omega$ is said coprime, if 
$(|G|,|\Omega|)=1$. It turns out that for coprime affine groups there
is an upper bound for the minimal base size: It was proved by D.\
Gluck K.\ Magaard \cite{gluck&magaard} that any such group has a base
of size at most. As the result of Seress is sharp,
the value of 95 can propably be improved.

Maybe the most examined case when $V$ is a finite vector space, 
$G\leq GL(V)$ is a solvable linear group and
$(|G|,|V|)=1$.  It was asked by I.\ M.\ Isaacs \cite{isaacs} 
whether there always exists a $G$-orbit in $V$ of size at least
$|G|^{1/2}$ for such groups. This follows immediately 
if we find $x,y\in V$ such that $C_G(x)\cap C_G(y)=1$, 
that is, a base of size two.
The existence of such vectors was confirmed by T.\ R.\ Wolf
\cite{wolf} in case of $G$ is supersolvable. Later, in a common work with
A.\ Moreto \cite{moreto} they solved this problem in case 
$|G|$ and $|V|$ are both odd. And S.\ Dolfi \cite{dolfi} proved that
it is enough to assume that  $G$ is odd.

The goal of this paper is to prove the following theorem, which completes
the remaining cases clear.
\begin{thm}\label{thm:main}
Let $V$ be finite vector space over the field of size $p$, where 
$p\neq 2$ is a prime, and let $G\leq GL(V)$ be a solvable linear group
with the assumption $(|G|,|V|)=1$. Then there exist $x,y\in V$ such that  
$C_G(x)\cap C_G(y)=1$.
\end{thm}
Using Hartley--Turull Lemma \cite{hartley} this yields

\begin{thm}
Let $G$ be a finite group acting faithfully on a finite group $K$ such that $(|G|,|K|)=1$. 
Then there exist $x,y\in K$ such that  
$C_G(x)\cap C_G(y)=1$.

\end{thm}

\section{Regular partitions for solvable permutation groups}\label{2bas:perm}
Throughout this section let $\Omega$ be a finite set and let
$G\leq\textup{Sym}(\Omega)$ be a solvable permutation group. For a subset
$X\sbs \Omega$ let $G(X)$ denote the set-wise stabilizer of $X$
in $G$-ben, that is, 
$G(X)=\{g\in G\;|\; gx\in X \textrm{ for all } x\in X\}$.
We say that the partition $\{\Omega_1,\Omega_2,\ldots,\Omega_k\}$ of
$\Omega$ is $G$-regular, if only the identity element of $G$ 
fixes all elements of this partition, i.e., if 
$\cap_{i=1}^k G(\Omega_i)=1$.\par
With the additional assumption that $G$ is a $p'$-group, one goal of
this section is to show a $G$-regular partition of $\Omega$ to at most
$p$ parts. Such a partition will be used in section 4 to reduce the
problem to primitive linear groups.
Moreover, our constructions for primitive permutation groups will be
useful even in the discussion of the primitvive linear case.
Since a primitive solvable permutation group is of affine type,
first we construct such partitions for affine groups. \par
\begin{thm}\label{thm:affin_perm}
Let $W$ be an $n$ dimensional vectorspace over the $q$-element field
($q$ is prime), and let $AGL(W)$ denote the full affine group acting on
$W$. Furthermore, let $G=W\rtimes G_0\leq AGL(W)$ be any subgroup.
If $\{\ud e_1,\ud e_2,\ldots,\ud e_n\}$ is a basis
of $W$, then, depending on $n$ and $q$, the 
following partitions are $G$-regular.
\begin{enumerate}
\item[]Case 1: $|W|\leq 4$\par
Take the trivial partition, that is, each element of the partition
consists of a single element.
\item[]Case 2: $n=1,\ q\geq 5$
\begin{gather*}
 \Omega_1=\{\ud 0\},\quad \Omega_2=\{\ud e_1\},\quad 
\Omega_3=V\setminus(\Omega_1\cup\Omega_2).
\end{gather*}
\item[]Case 3: $n\geq 2,\ q\geq 5$
\begin{gather*}
\Omega_1=\{\ud 0\},\\ 
\Omega_2=\{\ud e_1,2\ud e_1,\ud e_2,\ud e_3,\ldots,\ud e_n,
\ud e_1+\ud e_2,\ud e_2+\ud e_3,\ldots,\ud e_{n-1}+\ud e_n\},\\
\Omega_3=V\setminus(\Omega_1\cup\Omega_2).
\end{gather*}
\item[]Case 4: $n\geq 2,\ q=3$
\begin{gather*}
\Omega_1=\{\ud 0\},\ \Omega_2=\{\ud e_1\},\
\Omega_3=\{\ud e_2,\ud e_3,\ldots,\ud e_n,\ud e_1+\ud e_2,
\ud e_2+\ud e_3,\ldots,\ud e_{n-1}+\ud e_n\},\\
\Omega_4=V\setminus(\Omega_1\cup\Omega_2\cup\Omega_3).
\end{gather*}
\item[]Case 5: $n=3,\ q=2$
\begin{gather*}
\Omega_1=\{\ud 0\},\quad
\Omega_2=\{\ud e_1\},\quad
\Omega_3=\{\ud e_2\},\quad
\Omega_4=\{\ud e_3\},\\
\Omega_5=V\setminus(\Omega_1\cup\Omega_2\cup\Omega_3\cup\Omega_4).
\end{gather*}
\item[]Case 6: $n\geq 4,\ q=2$
\begin{gather*}
\Omega_1=\{\ud 0\},\quad
\Omega_2=\{\ud e_1\},\quad
\Omega_3=\{\ud e_2\},\\
\Omega_4=\{\ud e_3,\ldots,\ud e_n,
\ud e_3+\ud e_4,\ud e_4+\ud e_5,\ldots,\ud
e_{n-1}+\ud e_n,\ud e_3+\ud e_2,\ud e_n+\ud e_1\},\\
\Omega_5=V\setminus(\Omega_1\cup\Omega_2\cup\Omega_3\cup\Omega_4).
\end{gather*}
\item[]Case 7: $n=2,\ q=2$, and the order of $|G|$ is not divisible
by $3$\par
Let $\Omega_1=\{\ud 0\}$.
The action of $G(\Omega_1)$ on $V\setminus \Omega_1=\{\ud e_1,\ud
e_2,\ud e_1+e_2\}$ cannot be transitive, so it has a fix point
in $V\setminus \Omega_1$, say, $\ud e_1$. Then 
\begin{gather*}
\Omega_1=\{\ud 0\},\quad \Omega_2=\{\ud e_2\},\quad
\Omega_3=V\setminus(\Omega_1\cup\Omega_2).
\end{gather*} 
is $G$-regular.
\item[]Case 8: $n=3,\ q=2$, and the order of $|G|$ is not
divisible by $3$\par
Let $\Omega_1=\{\ud e_1,\ud e_2,\ud e_1+\ud
e_2\}$.  The action of $G(\Omega_1)$ on $\Omega_1$ cannot be
transitive, so it has a fix point in $\Omega_1$, say, $\ud
e_1$.\par
If $|G_0|$ is not divisible by $4$, then there exists an $\ud x\in
V\setminus(\Omega_1\cup\{\ud 0\})$, such that 
\begin{gather*}
\Omega_1=\{\ud e_1,\ud e_2,\ud e_1+\ud e_2\},\quad 
\Omega_2=\{\ud 0,\ud x\},\quad
\Omega_3=V\setminus(\Omega_1\cup\Omega_2).
\end{gather*} 
is $G$-regular.\par
Otherwise, we can assume that $G_0$ is contained in the group of upper
unitriangular matrices. In this case
\begin{gather*}
\Omega_1=\{\ud e_1,\ud e_3,\ud e_1+\ud e_3\},\quad 
\Omega_2=\{\ud e_2,\ud e_2+\ud e_3\},\quad
\Omega_3=V\setminus(\Omega_1\cup\Omega_2).
\end{gather*} 
is $G$-regular.
\item[]Case 9: $n\geq 4,\ q=2$, and the order of $|G|$ is not
divisible by $3$\par
Let $\Omega_1=\{\ud e_1,\ud e_2,\ud e_1+\ud
e_2\}$.  The action of $G(\Omega_1)$ on $\Omega_1$ cannot be
transitive, so it has a fix point in $\Omega_1$, say, $\ud
e_1$. Then the following partition is $G$-regular.
\begin{gather*}
\Omega_1=\{\ud e_1,\ud e_2,\ud e_1+\ud e_2\},\\
\Omega_2=\{\ud e_3,\ud e_4,\ldots,\ud e_n,\ud e_3+\ud
e_4,\ldots,\ud e_{n-1}+\ud e_n, \ud e_3+\ud e_2,\ud e_n+\ud e_1\},\\
\Omega_3=V\setminus(\Omega_1\cup\Omega_2).
\end{gather*}
\end{enumerate}
\end{thm}
\begin{proof} We show the $G$-regularity of the given partition 
only in Case \textsl{9}. In the remaining cases one can prove the same 
by using similar but simpler arguments.\par
Our first observation is that
$G(\Omega_1)$ fixes $\ud 0$, since $\Omega_1\cup \{\ud 0\}$
is the only $2$-dimensional affine subspace containing $\Omega_1$.
Hence $G(\Omega_1)\leq GL(W)$.
Let $g\in G(\Omega_1)\cap G(\Omega_2)$. We claim that $g(\ud
e_i)=\ud e_i$ for all $1\leq i \leq n$-re, that is, $g=1$.
First, $g(\ud e_2)=\ud
e_2$ or $g(\ud e_2)=\ud e_1+\ud e_2$ by our assumption to $\ud e_1$. 
In the second case $g(\ud e_3)\in\Omega_2$ and $g(\ud
e_2+\ud e_3)=\ud e_1+\ud e_2+g(\ud e_3)\in\Omega_2$. It is easy to
check that there is no $\ud x\in\Omega_2$ such that $\ud
e_1+\ud e_2+\ud x\in\Omega_2$. (Here we need $n\geq 4$). 
So $g(\ud e_2)=(\ud e_2)$.\par
To prove that $g(\ud e_k)=\ud e_k$ for all $3\leq k<n$ we use
induction to $k$. Assuming that $g(\ud e_i)=\ud
e_i$ for all $1\leq i<k<n$, it follows that $g(\ud e_k)$ and $g(\ud
e_{k-1}+\ud e_k)$ are elements of the set
\[
\Omega_2\setminus\left<\ud e_1,\ldots,\ud e_{k-1}\right>=\{\ud e_k,\ud
e_{k+1},\ldots,\ud e_n,\ud e_{k-1}+\ud e_k,\ldots,\ud e_{n-1}+\ud
e_{n},\ud e_n+\ud e_2\}.
\]
Since $g(\ud e_k)+g(\ud e_{k-1}+\ud e_k)=\ud e_{k-1}$, we have either
$g(\ud e_k)$ or $g(\ud e_{k-1}+\ud e_k)$ contains $\ud e_{k-1}$ with
non-zero coefficient. However, the only such element in the set
$\Omega_2\setminus\left<\ud e_1,\ldots,\ud e_{k-1}\right>$
is $\ud e_{k-1}+\ud e_k$. 
So either $g(\ud e_{k-1}+\ud e_k)=\ud e_{k-1}+\ud e_k$ or
$g(\ud e_k)=\ud e_{k-1}+\ud e_k$. In the latter case
$\ud e_k+\ud e_{k+1}\in \Omega_2$, since $k<n$, but $g(\ud e_k+\ud
e_{k+1})\not\in\Omega_2$, a contradiction. It follows that 
$g(\ud e_{k-1}+\ud e_k)=\ud e_{k-1}+\ud e_k$, 
which proves that $g(\ud e_k)=\ud e_k$. 
It remains to prove that $g(\ud e_n)=\ud e_n$.
It is clear that
\[
g(\ud e_n)\in\Omega_2\setminus\left<\ud e_1,\ud e_2,\ldots,\ud e_{n-1}\right>=
\{\ud e_n,\ud e_{n-1}+\ud e_n, \ud e_n+\ud e_1\}.
\]
If $g(\ud e_n)=\ud e_{n-1}+\ud e_n$, then
$g(\ud e_n+\ud e_1)=\ud e_{n-1}+\ud e_n+\ud e_1\not\in\Omega_2$. 
If $g(\ud e_n)=\ud e_{n}+\ud e_1$, then 
$g(\ud e_{n-1}+\ud e_n)=\ud e_{n-1}+\ud e_{n}+\ud e_1\not\in\Omega_2$.
Thus $g(\ud e_n)=\ud e_n$ also holds. So
$G(\Omega_1)\cap G(\Omega_2)=1$, that is, the given partition is $G$-regular.
\end{proof}
These constructions have some properties, which will be important
later. We summarize them in the following corollary.
\begin{cor}\label{kov:affin_perm}
If $W\leq G\leq AGL(W)$ is an affine group,
$p\geq 3$ prime, and $p$ does not divide the order of $G$, then there
exists a $G$-regular partition of $W$ into at most $p$ parts. 
Moreover, this partition has the following properties.
\begin{enumerate}
\item 
 In Case 1 the partition is trivial and it consists of at most $p-1$ parts.
 In any other case there is a part of ``unique size'', that is, a part
 $\Omega_i$ such that $|\Omega_i|\neq |\Omega_j|$ if $i\neq j$. 
\item
Apart from a few exceptions, there is a ``large'' part.
More precisely, the following inequalities holds:
\begin{gather}
\begin{array}{lll}
\textrm{In Case 2}:&|\Omega_2|=1<\frac{1}{4}|W|;&\\
\textrm{In Case 3}:&|\Omega_2|=2n<\frac{1}{4}5^n\leq\frac{1}{4}|W|;&\\
\textrm{In Case 4}:&|\Omega_3|+2=2n<\frac{1}{4}3^n=\frac{1}{4}|W|,
&\textrm{if }n\geq 3;\\
\textrm{In Case 5}:&|\Omega_4|=1<\frac{1}{4}|W|;&\\
\textrm{In Case 6}:&|\Omega_4|=2n-3<\frac{1}{4}2^n=\frac{1}{4}|W|,
&\textrm{if }n\geq 5.
\end{array}
\end{gather}
So, one of the above inequalities holds, unless 
$|W|=2,3,4,9$ or $16$. 
\end{enumerate}
\end{cor}
\begin{proof}
If $W$ is a vector space over the $q$-element field, then $(q,p)=1$, since
$G\geq W$. Now, if $p=3$, then $q\neq 3$, so one of the cases 
\textsl{1., 2., 3., 7., 8.,} or \textsl{9.} holds, and in these cases
the given partition is of 3 parts.
If $p\neq 3$, then $p\geq 5$. Even in the remaining cases the given
partition is of at most 5 parts. The remaining parts of the statement
can be easily checked.
\end{proof}
Using the first part of this Corollary we can prove the existence
of the wanted $G$-regular partition for any solvable $p'$-group.
\begin{thm}\label{thm:perm}
Let $G\leq\textup{Sym}(\Omega)$ be a solvable permutation group, 
Assuming that the order of $G$ is not divisible by $p$, there exists
a $G$-regular partition of $\Omega$ into at most $p$ parts.
\end{thm}
Before proving this, first we give an alternative form of this
statement, which will be easier to handle. Besides that, from this
form it should be clearer what is the connection between finding
$G$-regular partition for a permutation groups and finding 
a two-element base for a linear group.
If $\Omega=\{1,2,\ldots,n\}$, then we have a natural inclusion
$\textup{Sym}(\Omega)\rightarrow GL(n,p)$. 
Hence $\textup{Sym}(\Omega)$ acts naturally on $\FF p^n$. 
If we have a partition of $\Omega$ into at most $p$ parts, then 
we can color the elements of the partitions by the elements of $\FF
p$, that is, there is an $f:\Omega\rightarrow\FF p$ such that
$x,y\in\Omega$ are in the same part of the partition if and only if
$f(x)=f(y)$. 
Thus, Theorem \ref{thm:perm} is clearly equivalent to the following theorem.
\begin{thm}\label{thm:perm_var}
If $G$ is a solvable permutation group of degree $n$, and $p$ does not
 divide the order of $G$, then there is an $(a_1,a_2,\ldots,a_n)\in\FF p^n$
 vector, such that only the identity element of
 $G$ fixes this vector.
\end{thm}
\begin{proof}
Although we do not deal with the case $p=2$, we note that
this follows from a Theorem of D.\ Gluck\cite{gluck}. A direct short
proof is given by H.\ Matsuyama \cite{matsuyama}. Thus, let in the
following $p\geq 3$.\par
If $G$ is primitive permutation group, then Corollary 
\ref{kov:affin_perm} guarantees the existence of such a vector (or partition).
In the following let $G$ be a transitive, but not primitive group.
Then there are blocks $\Delta_i,\ 1\leq i\leq k$, such that
$1<|\Delta_i|<|\Omega|$, 
$\Omega=\Delta_1\cup\Delta_2\cup\ldots\cup\Delta_k$ is a partition, 
and $G$ permutes the $\Delta_i$ sets transitively.
We can assume that $|\Delta_i|$ is as small as possible.
Let $H_i=G(\Delta_i)$ and $N=\cap_{i=1}^k H_i$. Now, $G/N$ 
transitively on the set
$\widetilde{\Omega}=\{\Delta_1,\Delta_2,\ldots,\Delta_k\}$. 
Using induction to $|\widetilde{\Omega}|$ 
we get a vector $(a_1,a_2,\ldots,a_k)\in\FF p^k$
such that only the identity element of $G/N$ fixes this vector.\\
On the other hand, $H_i/C_{H_i}(\Delta_i)$ acts primitively on 
$\Delta_i$ for all $1\leq i\leq k$, and these groups are all conjugate
in $G$; in particular, they are permutation isomorphic.
Thus, for each $i$ we can find a $H_i/C_{H_i}(\Delta_i)$-regular
partition of $\Delta_i$ by Corollary
\ref{kov:affin_perm} and these partitions are essentialy the same
for $1\leq i\leq k$.\\
If the first case holds in part \textsl{1} of the above Corollary, 
then $|\Delta_i|\leq p-1$. In this case let us choose an
$A\subset \FF p$ subset such that $|A|=|\Delta_i|$, and 
let $f_i:\Delta_i\rightarrow A+a_i=\{a+a_i\;|\;a\in A\}$ be
a bijection for every $1\leq i\leq k$.\\
If the second case holds in part \textsl{1} of the above Corollary, 
then for each $i$ 
let us choose $X_i\sbs \Delta_i$ part of the partition of $\Delta_i$,
such that it is of unique size. If there would be more than one such part,
then we only need to pay attention that the size of each $X_i$ must
be the same. Now, let the finction $f_i:\Delta_i\rightarrow\FF p$ 
be defined as a coloring of the partition of $\Delta_i$ such that 
$f_i(X_i)=a_i$.\\
Finally, let the function $f:\Omega\rightarrow \FF p$ be defined as
\[
f(x)=f_i(x), \textrm{\quad if\quad }x\in\Delta_i.
\]
The essence of this construction is that the distribution of $f_i$
determines $a_i$. 
Hence, if
$g\in G$ fixes the vector $(f(1),f(2),\ldots,f(n))\in\FF p^n$, then
$gN$ fixes the vector $(a_1,a_2,\ldots,a_k)\in\FF p^k$, so
$g\in N$ and $g(\Delta_i)=\Delta_i$ for each $1\leq i\leq k$. 
Finally, from the construction of the $f_i$-s we get
$g\in\cap_{i=1}^k C_{H_i}(\Delta_i)=1$.
\end{proof}
\noindent\textbf{Remark}

\medskip\noindent
It was proven by Á.\ Seress \cite[Theorem 1.2.]{seress} that
for any $G\leq \textrm{Sym}(\Omega)$ solvable permutation group
there always exists a $G$-regular partition of $\Omega$ into at most
five parts.
\bigskip\\
Our last result concerning permutation groups is showing regular partitions
to groups of affine type with ``mixed characteristic''.
This will play a role in the discussion of primitive linear groups.
\begin{thm}\label{thm:kevert}
For each $1\leq i\leq k$ let $W_i$ be a finite vector space over the 
$p_i$-element field, where $p_1<p_2<\ldots<p_k$ and $k\geq 2$. 
Furthermore, let $\oplus_{i=1}^k W_i\leq G\leq
AGL(W_1)\times AGL(W_2)\times\cdots \times AGL(W_k)$ acting 
on $W=W_1\oplus W_2\oplus\cdots\oplus W_k$ in the natural way.
Then there exists a $G$-regular 
$W=\Omega_1\cup\Omega_2\cup\Omega_3$ partitition such that $\Omega_1=\{\ud
0\}$ and $|\Omega_2|<\frac{1}{4}|W|$.
\end{thm}
\begin{proof}
For each $1\leq i\leq k$ let $\ud e_{i,1},\ud e_{i,2},\ldots, \ud
e_{i,n_i}$ be a basis of $W_i$, where $n_i=\dim W_i$. 
To show a suitable $\Omega_2$ we use the cases 1-6 of Theorem 
\ref{thm:affin_perm}. We saw that there are $\Omega_i^*\sbs W_i$
subsets such that $G(\ud 0)\cap G(\Omega_i^*)=1$ for $p_i\geq 5$, or
$|W_i|\leq 3$,
$G(\ud 0)\cap G(\Omega_i^*)\cap G(\ud e_{i,1})=1$ for $p_i=3$ or $|W_i|=4$
and $G(\ud 0)\cap G(\Omega_i^*)\cap G(\ud e_{i,1})\cap G(\ud e_{i,2})=1$ for
$p_i=2, n_i\geq 3$. Now, let $\Omega_2$ be defined as
\[
\begin{array}{ll}
\{\ud e_{1,1}+\ud e_{2,1},\ud e_{1,1}+2\ud e_{2,1},
\ud e_{1,2}+\ud e_{2,1}\}\cup
\Omega_1^*\cup \Omega_2^*\cup\ldots\cup\Omega_k^*,&\textrm{if}\ 
p_1=2, n_1\geq 3;\\
\{\ud e_{1,1}+\ud e_{2,1}\}\cup
\Omega_1^*\cup \Omega_2^*\cup\ldots\cup\Omega_k^*,&\textrm{otherwise.}
\end{array}
\]
Now, $G(\ud 0)$ is a subgroup of the automorphism group of $\oplus
W_i$, so it fixes each $W_i$. Thus, if $g\in G(\ud 0)\cap
G(\Omega_2)$, then $g$ fixes each $W_i\cap \Omega_2=\Omega^*_i$, and
it permutes the three (or one) exceptional elements. Using that 
$g(e_{1,1}),g(e_{1,2})\in W_1$, $g(e_{2,1})\in W_2$, we get $g$ fixes
also these elements. Hence $g$ acts trivially on every $W_i$, so
$g=1$, and we found a $G$-regular partition. \\
Let $l=n_1+\ldots+n_k$. Then $|W|\geq 2^{l-1}3$, since $k\geq 2$. 
We saw that $|\Omega_i^*|\leq 2n_i$, and $|\Omega_1^*|\leq 2n_1-3$, if 
$p_1=2,\ n_1\geq 3$. It follows that $|\Omega_2|\leq
1+2l<\frac{1}{4}2^{l-1}3\leq \frac{1}{4}|W|$ holds, unless $l\leq 4$. 
Now, assume that $l\leq 4$.
If each $n_i\leq 2$, then let $\Omega_2=\{\sum_i e_{i,1}\}\cup
\{e_{j,2}\,|\, n_j=2\}$, which is clearly $G(\ud 0)$-regular, and for which
$|\Omega_2|=l-1<\frac{1}{4}2^{l-1}3\leq 1/4|W|$ holds. 
In case of $n_1=3,\ p_1=2,\ n_2=1$, we have
$|\Omega_1^*|=|\Omega_2^*|=1$, so $|\Omega_2|=3+1+1<\frac{1}{4}2^3
3\leq\frac{1}{4}|W|$. Finally, if $|W|=p^3q$ for some $p\neq 2,q$
primes, then $|\Omega_2|\leq 1+6\leq \frac{1}{4}3^3 2\leq\frac{1}{4}|W|$.
\end{proof}
\section{Primitive linear groups}\label{2bas:prim}
In the following let $V\simeq\FF p^n$ be a finite vector space and let 
$G\leq GL(V)\simeq GL(n,p)$ be a solvable linear group such that
$(|G|,p)=1$. In this section we assume that $G$ is a primitive linear
group, that is, there does not exists a
\[
V=V_1\oplus V_2\oplus\ldots\oplus V_t
\]
proper decomposition of $V$, such that $G$ permutes the $V_i$
subspaces. (We deal with imprimitive linear groups in section
\ref{2bas:imprim}.)
In order to find vectors $x,y\in V$ such that 
$C_G(x)\cap C_G(y)=1$, we can clearly assume that $G$ is maximal (with
respect to inclusion) among the
solvable $p'$-subgroups of $GL(V)$.
The main idea of our construction is the following: Using that the
Fitting subgroup of $G$ (denoted by $F$) has a very special structure,
we show the existence of a basis of $V$ such that every element of $F$
is ``almost'' monomial in this basis. 
Next, we choose $x$ in such a way that 
$C_G(x)$ is also ``almost'' monomial subgroup in this basis. 
Now, the permutation part of $F$ defines a linear space on this
special basis, and the permutation part of $C_G(x)$ acts on
this space as a linear group. Hence we can use the constructions given in
Theorem \ref{thm:affin_perm} and in Theorem \ref{thm:kevert}
to find a suitable $y$. 
\subsection{The structure of the Fitting subgroup}\label{subsection:fitting}
If $G\leq GL(V)$ is a maximal solvable $p'$-subgroup, then it is a
 $p'$-Hall subgroup of some $H\leq GL(V)$ maximal solvable subgroup.
Some relevant properties of such groups can be found in
\cite[Proposition 2.1]{p3}, \cite[Lemma 2.2]{seress} 
and in \cite[\S 19--20]{sup}. We collect them in the following theorem.
\begin{thm}\label{thm:maxprim}
Let $H\leq GL(n,p)$ be a maximal solvable primitive group. 
Then $H$ contains a unique maximal abelian normal subgroup, denoted by 
$A$. Furthermore, let $C=C_H(A)$ and $F=\textrm{Fit}(C)$, the Fitting
subgroup of $C$. Now, $A\leq F\leq C\leq H$ are all normal subgroups
of $H$, which have the following properties.
\begin{enumerate}
\item $A$ is cyclic and $|A|=p^a-1$ for some $a$.
\item The linear span of $A$ is isomorphic to the field $\FF {p^a}$.
\item The action of $H/C$ on $A$ gives us an inclusion 
$H/C\hookrightarrow\FF{p^a})$. 
\item $F=AP_1P_2\ldots P_k$, where $P_i$ is an
extraspecial $p_i$-group of order $p_i^{2e_i+1}$ for each $i$.
Furthermore, $Z(P_i)=A\cap P_i$, and $A$ contains all the $p_i$-th
roots of unity. If $p_i>2$, then the exponent of $P_i$ is $p_i$. 
\item Let $e=\prod p_i^{e_i}$. Then $n=ea$. 
\item $C$ is included in $GL(e,p^a)$.
\item $F\leq GL(e,p^a)$ gives an irreducible representation of $F$.
\end{enumerate}
\end{thm}
Now, if $G\leq H$ is a $p'$-Hall subgroup of $H$, then we claim that
 $A\leq F\leq G$. Indeed, $A$ and $F$ both are normal $p'$-subgroups o
 $H$, so they are contained in a $p'$-Hall subgroup of $H$. Since the 
$p'$-Hall subgroups are all conjugate, they are contained in $G$, too.
Hence we can use the above theorem to $G$.
By the next lemma we can assume that $C_G(A)=G$.
\begin{lem}
Let $x,y\in V$ such that $C_C(x)\cap C_C(y)=1$. 
Then for some $\gamma\in A\cup \{0\}=\FF {p^a}$  we have
$C_G(x)\cap C_G(y+\gamma x)=1$.
\end{lem}
\begin{proof}
For any $g\in G$ let $\sigma_g\in\textrm{Aut\;}(\FF {p^a})$ denote
the action of $gC$ on $\FF {p^a}$ by part \textit{3} of Theorem 
\ref{thm:maxprim}.
For all $\alpha\in\FF {p^a}$ let the subgroup
$H_\alpha=C_G(x)\cap C_G(y+\alpha x)\leq G$.
Our aim is to prove that $H_\alpha=1$ for some $\alpha\in\FF {p^a}$.\\
Let $g\in H_\alpha$. Thus, $g(x)=x$ and $y+\alpha x=g(y+\alpha x)=g(y)+
\alpha^{\sigma_g}x$. Hence $g(y)=y+(\alpha-\alpha^{\sigma_g})x$.
If $g\in \left<\cup H_\alpha\right>$, then $g$ is the product of elements
from several $H_\alpha$'-s. It follows that $g(y)=y+\delta x$ for a  
$\delta\in\FF {p^a}$.\\
We claim that $\left<\cup H_\alpha\right>\cap C=1.$
Let $g\in \left<\cup H_\alpha\right>_G\cap C$.
On the one hand, the action of $g$ on $V$ is $\FF {p^a}$-linear, since 
$g\in C=C_G(A)$. 
On the other hand, $g(x)=x$ and $g(y)=y+\delta x$ for a 
$\delta\in\FF {p^a}$ by the previous part.
If $g^n=1$, then $y=g^n (y)=y+n\delta x$, so $n\delta=0$. 
Using that $|G|$ is coprime to $p$, we get $n$ is not divisible by $p$, 
hence $\delta=0$. Therefore, $g(y)=y$ and $g\in C_C(x)\cap
C_C(y)=1$.\\
Since $G/C\leq \textrm{Aut\;}(\FF {p^a})$, for any  
$g\neq h\in\cup H_\alpha$ we have $\sigma_g\neq\sigma_h$. Furthermore,
the subfields of $\FF {p^a}$ fixed by $\sigma_g$ and $\sigma_h$ 
are the same if and only if $\left<g\right>_G=\left<h\right>_G$.\\
If $g\in H_\alpha\cap H_\beta$, then 
$g(y)=y+(\alpha-\alpha^{\sigma_g})x=y+(\beta-\beta^{\sigma_g})x$, so
$\alpha-\beta$ is fixed by $\sigma_g$.\\
Let $K_g=\{\alpha\in\FF {p^a}\;|\;g\in H_\alpha\}$.
The previous calculation shows that $K_g$ is an additive coset
of the subfield fixed by $\sigma_g$, so $|K_g|=p^d$ for some $d|a$. 
Since for any $d|a$ there is a unique $p^d$-element subfield of $\FF
{p^a}$, we get $|K_g|\neq |K_h|$ unless the subfields fixed by
$\sigma_g$ and $\sigma_h$ are the same. 
As we have seen, this means $\left<g\right>_G=\left<h\right>_G$. 
Consequently, $|K_g|\neq |K_h|$ unless $K_g=K_h$.
Hence we get the following 
\[
\big|\!\!\!\!\bigcup_{g\,\in\,\cup H_\alpha\setminus\{1\}}\!\!\!\! K_g\big|
\leq\sum_{d|a,\;d<a}p^d\leq 
\sum_{d<a}p^d=\frac{p^a-1}{p-1}<p^a.
\]
So there is a $\gamma\in\FF{p^a}$ which is not contained in $K_g$ for any
$g\in\cup H_\alpha\setminus\{1\}$. This exactly means
$H_\gamma=C_G(x)\cap C_G(y+\gamma x)=1$.
\end{proof}
Henceforth, it is enough to find suitable 
$x,y\in V\simeq\FF {p^a}^e$ vectors for such $G\leq GL(e,p^a)$
solvable $p'$-groups,
which have $A\leq F\leq G$ normal subgroups, $A$ consists of scalar
matrices, and parts \textsl{4,5,7} of Theorem \ref{thm:maxprim} 
holds to $F$.\\
Observe that for each $p\neq 2$ prime the \textsl{4}th part of Theorem
\ref{thm:maxprim} determines the isomorphic type of the $p$-Sylow
subgroup of $F$, since there are two types of extraspecial groups 
of order $p^{2d+1}$ for any $p$: For $p\neq 2$ one of them has exponent
$p$, the other one has exponent $p^2$. However, for $p=2$ both of them
has exponent 4. In this case one of them is the central product of
$d$ copies of dieder groups $D_4$, the other one is the central
product of a quaternion group $Q$ and $d-1$ copies of $D_4$. This
gives us two possible isomorphism type to $F$. We say that $F$ is
monomial, if in the above decomposition of $F$ either each $p_i\neq 2$
(that is, $e$ is odd), or the occuring extraspecial subgroup in $F$,
say $P_1$, is a central power of $D_4$. Otherwise, we say that $F$
is not monomial. (The explanation of our term ``monomial'' is that 
in the first case we can choose a basis such that written in this
basis every element of $F$ will be monomial matrix.)
\subsection{Finding $x,y\in V$ in case $F$ is monomial}
Let in the following $F\nor G\leq GL(V)\simeq GL(e,p^a)$, where $F$
monomial, that is, the extraspecial subgroup of $F$ occuring in
part \textsl{4} of Theorem \ref{thm:maxprim} is a central power of $D_4$
(maybe trivial). 
The next theorem help us to find a ``good'' basis to $F$.
\begin{thm}\label{thm:fitting_mon}
With the above assumptions the following hold to $F\nor GL(V)$:
  \begin{enumerate}
  \item There is a decomposition $F=D\rtimes S$ such that $D=A\times D_0$, 
    and
    \[
	D_0\simeq S\simeq Z_{p_1}^{e_1}\times
	Z_{p_2}^{e_2}\times\ldots\times Z_{p_k}^{e_k}.
    \]
  \item There is a basis $\monomialbasis\in V$ such that written 
  in this basis $D$ consists of diagonal matrices and 
  $S$ regularly permutes the elements of this basis.
  \item The subspaces $\left<\ud u_i\right>,\ 1\leq i\leq e$ are
  all the irreducible representations of $D_0$ over $\FF {p^a}$,
  and they are pairwise non-equivalent.
  \item If $g\in D_0$, then $g$ contains all of the $o(g)$-th
  roots of the unity with the same multiplicity.
  \end{enumerate}
\end{thm}
\begin{proof}
It is well-known that any extraspecial $p$-group is the central
product of non-abelian groups of order 
$p^3$. Taking our restriction to $P_i$ in to account in case $p_i=2$
and using that the exponent of $P_i$ is $p_i$ for $p_i>2$ we can find 
generators
\[
P_i=\left<x_{i,1},x_{i,2},\ldots,x_{i,e_i}, y_{i,1},y_{i,2},
\ldots y_{i,e_i},z_i\right>,
\]
such that any generator is of order $p_i$,
$Z(P_i)=\left<z_i\right>$, and $[x_{i,l},y_{i,l}]=z_i$ for all
$1\leq l\leq e_i$, any other pair of generators are commuting.
Now, let $D_i=\left<x_{i,1},x_{i,2},\ldots,x_{i,e_i}\right>$, and
$S_i=\left<y_{i,1},y_{i,2},\ldots,y_{i,e_i}\right>$.
Finally, let
\[
  D=A\times D_1\times D_2\times\cdots\times D_k\textrm{\quad és\quad}
  S=S_1\times S_2\times\cdots\times S_k. 
\]
Now, it should be clear that the decomposition $F=D\rtimes S$ fulfill
the requirement \textsl{1}. (Although we did not fix $D_0$-t yet!)
Using part \textsl{4} of Theorem \ref{thm:maxprim}, we get
$A=\FF{p^a}^*$ contains all of the $\exp(D)$-th roots of unity, hence
every irreducible representation of $D$ over $\FF{p^a}$ is one dimensional.
Fix an $\ud u_1\in V$ in such a way that $\FF {p^a}\ud u_1$ is a
$D$-invariant subspace. Choosing $D_0=C_D(\ud u_1)$ we have $D=A\times
D_0$.\\
Now, let the basis $\{\mb\}$ be defined as the set
$\{s(\ud u_1)\,|\, s\in S\}$. First of all, $e=|S|=\dim V$. As 
$D\nor F$, it follows that
 $Ds(\ud u_1)=sD(\ud u_1)$, so $\FF {p^a}s\ud u_1$ is also a
$D$-invariant subspace for all $s\in S$. Hence
$\left<\ud u_1,\ud u_2,\ldots,\ud u_e\right>$ is an $F=DS$-invariant
subspace, so it is equal to $V$ by part  \textsl{7} of Theorem 
\ref{thm:maxprim}. Therefore, $\{\ud u_1,\ud u_2,\ldots,\ud u_e\}$
is indeed a basis of $V$. From our construction \textsl{2} clearly follows.\\
The \textsl{3}rd part of the statement follows easily from the fact 
$C_S(D_0)=1$. Let $\ud u_i=s(\ud u_j)$, where $1\neq s\in
S$. Furthermore, let $d\in D_0$ such that the scalar matrix 
$[d,s]\neq 1$. Then
\[
d_{jj}\ud u_j=d(\ud u_j)=ds(\ud u_i)= sd[d,s](\ud u_i)=[d,s](d_{ii}\ud u_j),
\]
so $d_{jj}\neq d_{ii}$, which proves that these representations are
pairwise non-isomor\-phic. The statement that these representations give us
all the irreducible representations of $D_0$ follows from the fact $|D_0|=e$.\\
Finally, in view of the last statement, part 
\textsl{4} is just a special case of a more general statement to any 
$A$ finite abelian group and to the groups of linear characters of $A$
over $K$ with the assumptions $(|A|,|K|)=1$ and $K$ contains all of the 
$\exp(A)$-th roots of unity.
\end{proof}
In the following we fix a basis $\mb$, which fulfill the requirements
of the above theorem. With respect to this basis, we identify 
$GL(V)$ with the matrix group $GL(e,p^a)$.
Thus, $F=DS\nor G\leq GL(e,p^a)$, where $D$ is the group of
diagonal matrices in $F$
and $S$ is the group of permutation matrices in $F$
acting regularly on the selected basis. 
Furthermore, $D=A\times D_0$, where $D_0=C_D(\ud u_1)=C_F(\ud u_1)$.

To find a base $x,y\in V$ we write them as a linear combination of the
matrices $\mb$ in such a way that $x$ contains only a few 
(one or three) $\ud u_i$ with non-zero coefficients, while $y$ a lot
of them.

Our next lemma collects some consequences of the choice $x=\ud u_1$: 
\begin{lem}\label{thm:monlem}
Let $g\in G$ be any group element fixing $g(\ud u_1)=\ud u_1$. Then
\begin{enumerate}
\item $D_0^g=D_0$, and $g$ is a monomial matrix.
Hence there exists a $g=\delta(g)\pi(g)$ decomposition of $g$ to a
diagonal matrix $\delta(g)$ and to a permutation matrix $\pi(g)$.
\item $\pi(g)$ normalizes $S$, that is, $S^{\pi(g)}=S$.
\item Both $\delta(g)$ and $\pi(g)$ normalize $F$, so
$F=F^{\delta(g)}=F^{\pi(g)}$. Moreover, $[\delta(g), S]\leq D$.
\item If $\delta(g)\neq 1$, then the numbers of $1$'-s 
in the main diagonal of $\delta(g)$ is at most $\frac{3}{4}e$.
\end{enumerate}
\end{lem}
\begin{proof}
The statement $D_0^g=D_0$ follows from the fact $D_0=C_F(\ud 
u_1)\nor C_G(\ud u_1)$. Consequently, $g$ permutes the homogeneous
components of the $D_0$-module $V$. By part \textsl{3} of Theorem 
\ref{thm:fitting_mon}, these homogeneous components are just the
one-dimensional subspaces $\left<\ud u_i\right>,\ \ 1\leq
i\leq e$. These means that $g$ is a monomial matrix. 
Of course, a monomial matrix $g$ has a unique decomposition
$g=\delta(g)\pi(g)$, and part \textsl{1} is proved.

For any $s\in S$ we have
\[
s^g=\pi(g)^{-1}\delta(g)^{-1}s\delta(g)\pi(g)=
\pi(g)^{-1}([\delta(g),s^{-1}]s)\pi(g)
=[\delta(g),s^{-1}]^{\pi(g)}s^{\pi(g)}
\]
is an element of $F$. The expression
$[\delta(g),s^{-1}]^{\pi(g)}$ on the right-hand side is diagonal, 
while $s^{\pi(g)}$ is permutation matrix, so both of them are elements
of $F$. However, any permutation matrix in $F$ is contained in $S$, so
$s^{\pi(g)}\in S$, and \textsl{2} follows.

Both $g$ and $\delta(g)$ normalize $D$, hence
$\pi(g)=\delta(g)^{-1}g$, too. 
We have seen that $\pi(g)$ normalizes $S$, so it normalizes also
$F=DS$. We get $\delta(g)=g\pi(g)^{-1}$ also normalizes 
$F$. Finally, the statement $[\delta(g), S]\leq D$ follows from the
fact that the commutator of a permutation matrix by a diagonal matrix
is also diagonal. So \textsl{3} holds.

If $\delta(g)\neq 1$, then $\delta(g)$ is not a scalar matrix, so 
there exists an $s\in S$ such that $[\delta(g), s]\not =1$. Now,
$[\delta(g), s]\in D\setminus \{ 1\}$, so, using part \textsl{4} of
Theorem  \ref{thm:fitting_mon}, we get the number of $1$'-s in the
main diagonal of $[\delta(g), s]$ is at most $\frac{1}{2}e$. 
This cannot be true if the number of $1$'-s in $\delta(g)$ is more
than $\frac{3}{4}e$. We are done.
\end{proof}
We saw that for any $g\in C_G(\ud u_1)$ there is a unique
decomposition $g=\delta(g)\pi(g)$. The map $\pi:g\rightarrow \pi(g)$
gives us a homomorphism from $C_G(\ud u_1)$ into the group of
permutation matrices.

By part \textsl{2} of Lemma \ref{thm:monlem} Lemma, $\pi(C_G(\ud
u_1))$ normalizes $S$, so it acts on $S$ by conjugation, which defines
a $\pi(C_G(\ud u_1))\rightarrow \textrm{Aut}(S)$ homomorphism. 
In fact, this homomorphism is an inclusion, since 
$C_G(\ud u_1)\cap C_G(S)=1$. Therefore,
$\pi(C_G(\ud u_1))\leq \textrm{Aut}(S)\simeq GL(e_1,p_1)
\times GL(e_2,p_2)\times\cdots GL(e_k,p_k)$.

This is usefull to us, because we can apply Theorems 
\ref{thm:affin_perm} and \ref{thm:kevert}, to find a $\pi(C_G(\ud
u_1))$-regular partition of $S$. Moreover, we do not need to fix the
zero element of $S$ (that is, the identity matrix); 
we already fixed it by choosing $x=\ud u_1$. 
Since $S$ acts on the basis $W=\{\ud u_1,\ud u_2,\ldots,\ud u_e\}$ regularly,
using the bijection $s\rightarrow s(\ud u_1)$ we can define a partition 
$W=\{\ud u_1\}\cup\Omega_2\ldots\cup\Omega_l$, 
which is also $\pi(C_G(\ud u_1))$-regular. 
\subsubsection{Case $e\neq 2^t$}
In the following we will assume that $|D_0|=|S|=e$ is not a
2-power. In every such case let $x=\ud u_1$.
By the last paragraph, we have a $\pi(C_G(\ud u_1))$ regular partition $W=\{\ud
u_1\}\cup\Omega_2\ldots\cup\Omega_l$. Let $\alpha\in \FF {p^a}$ be a generator
element of the multiplicative group of $\FF {p^a}$. Now,
$o(\alpha)=|A|\geq 6$, since $|A|$ is even ($p\neq 2$) and every prime
divisor of $e$ divides $|A|$. 
\begin{thm}\label{thm:mon_eneq2^k}
With the above notations let $y$ be defined as follows
\[
\begin{array}{ll}
\textrm{For }e\neq 3^k:& \displaystyle{y=0\cdot\sum_{\ud
u_i\in\Omega_2} \ud u_i+1\cdot\sum_{\ud u_i\in\Omega_3} \ud u_i,}\\
\textrm{For }e=3^k,\ k\geq 2:&\displaystyle{ y=\alpha\cdot\sum_{\ud
u_i\in\Omega_2} \ud u_i+0\cdot\sum_{\ud u_i\in\Omega_3} \ud u_i+
1\cdot\sum_{\ud u_i\in\Omega_4} \ud u_i,}\\
\textrm{For  }e=3:& y=\displaystyle{\alpha\ud u_2+\ud u_3.}
\end{array}
\]
Then $C_G(x)\cap C_G(y)=1$.
\end{thm}
\begin{proof}Let $g\in C_G(x)\cap C_G(y)$. 
Since $g$ fixes $\ud u_1=x$, $g$ is a monomial matrix 
by Lemma \ref{thm:monlem}, so we have a decomposition $g=\dg\pg$.

Our first observation in case $e\neq 3^k$ is that $\pi(g)$ fixes the
subset $\Omega_2\sbs W$ részhalmazt. To see this, notice that if the
monomial matrix $g$ fixes $y$, then $\pi(g)$ permutes the basis elements
appearing in $y$ with zero coefficients between each other.
So $\pi(g)$ fixes both $\ud u_1\cup\Omega$ and $\ud u_1$ (since $g$ does),
therefore it fixes $\Omega_2$. Since $W=\{\ud
u_1\}\cup\Omega_2\cup\Omega_3$ is a $\pi(C_G(\ud u_1))$-regular
partition, we get $\pi(g)=1$. Hence $g=\delta(g)$ is a diagonal
matrix. If $g_{ii}$ denote the $i$-th element of the main diagonal of
$g$, then $g\in C_G(y)$ holds only if $g_{ii}=1$ for all $\ud
u_i\in\Omega_3$. Since $e$ is neither $2$-power nor $3$-power, we can apply
Theorem \ref{thm:kevert} and the second part of Corollary 
\ref{kov:affin_perm} to get $|\Omega_2|<\frac{1}{4} e$.
Using part \textsl{4} of Lemma \ref{thm:monlem} it follows that $g=1$.

In case of $e=3^k,\ k\geq 3$ we see that $\pg$ fixes the subset 
$\Omega_3\sbs W$, since these elements occur with non-zero coefficient
$0$ in $y$. (not counting $x=\ud u_1$ which is already fixed by $g$.)
However, in this case it is possible that $\pi(g)$ takes the unique element of
$\Omega_2$ into an element of $\Omega_4$. Of course, in that case it
takes an element of $\Omega_4$ into the element of
$\Omega_2$. This results the appearance of an $\alpha$ and an
$\alpha^{-1}$ in the main diagonal of $\dg$. It follows that the
number of $\neq 1$ elements in the main diagonal of $\dg$ is at most 
$|\Omega_2|+2$, which is less than $\frac{1}{4}e$ by Corollary 
\ref{kov:affin_perm}, if $e\neq 9$.
By part \textsl{4} of Lemma \ref{thm:monlem} we get $\dg=1$, 
hence $\pi(g)$ also fixes the unique element of $\Omega_2$,
so $g=\pi(g)=1$.

It remains to examine the cases $e=9$ and $e=3$. 
In case of $e=9$ we have $y=\alpha\cdot \ud u_i+0\cdot \ud u_j+1\cdot
\sum_{k\neq i,j,1}\ud u_k$. Then $\pg$ fixes $\ud u_j$.
If $\pi(g)$ fixes also $\ud u_i$, then $\pi(g)=1$. 
In this case the only not necessarily $1$ element in the main diagonal
of $g=\delta(g)$ is $g_{jj}$. Using part \textsl{4} of Lemma 
\ref{thm:monlem} we get $g=1$.
If $\pi(g)$ does not fixes $\ud u_i$, then in the main diagonal of
$\dg$ there are an $\alpha$ and an $\alpha^{-1}$,
possibly $\dg_{jj}\neq 1$, any other element is $1$. Since $S$
acts regularly on $W$, we can choose an element $s\in S$ which takes
the bases element corresponding to $\alpha^{-1}$ into the bases element
corresponding to $\alpha$. Then, in the main diagonal of $[\dg,s]$ 
appear an $\alpha^2\neq 1$ and at least four $1$'-s. 
However, there is no such an element in $D=A\times
D_0$ by part \textsl{4} of Theorem \ref{thm:fitting_mon}, 
contradicting to part \textsl{3} of Lemma \ref{thm:monlem}.

Finally, let $e=3$. If $g\in C_G(x)\cap C_G(y)$ is diagonal, 
then clearly $g=1$. Otherwise,
\[
\dg=
{\s
\begin{pmatrix}
1&0&0\\
0&\alpha&0\\
0&0&\alpha^{-1}
\end{pmatrix}
}
\textrm{\ and\ }
[\dg,s]=
{\s
\begin{pmatrix}
\alpha&0&0\\
0&\alpha^{-2}&0\\
0&0&\alpha
\end{pmatrix}
},
\textrm{\ for\ }
s=
{\s
\begin{pmatrix}
0&0&1\\
1&0&0\\
0&1&0
\end{pmatrix}
}\in S.
\]
Since $o(\alpha)\geq 6$ we get $\alpha\neq \alpha^{-2}$, 
so $[\dg,s]\notin D$ by part \textsl{4} of Theorem
\ref{thm:fitting_mon}, which is impossible by part \textsl{3} of Lemma
\ref{thm:monlem}.
\end{proof}
\subsubsection{Case $e=2^t$}
Keeping the assumption that $F$ is monomial, now we handle the case
$e=2^k$ for some $k$. We note that in case of $e\leq 128$ we could give
similar constructions as we did in Theorem \ref{thm:mon_eneq2^k}. 
However, for a more uniform discussion we alter these constructions a
bit, so it will be adequate even in smaller dimensions. The point of
our modification is that we do not choose $x$ as a bases element this time,
rather as a linear combination of exactly three bases vectors.
Although this effects that $C_G(x)$ will not be monomial any more, 
but we can cure this problem by a good choice of $y$.

In case $e=2$ any bases will be obviously good, let for example 
$x=\ud u_1, y=\ud u_2$.
Now, we analyze the case $e=4$. According to Theorem \ref{thm:fitting_mon},
we choose a bases $\ud u_1, \ud u_2,\ud u_3,\ud u_4\in V$.
Now, $F=AD_0S$, where the Klein groups $D_0=\left<d_1, d_2\right>$ 
and  $S=\left<s_1,s_2\right>$ are generated (independently from the
base field) by the matrices:
\begin{gather*}
d_1=
\begin{pmatrix}
1& &  &  \\
 &1&  &  \\
 & &-1&  \\
 & &  &-1
\end{pmatrix},
d_2=
\begin{pmatrix}
1& &  &  \\
 &-1&  &  \\
 & &1&  \\
 & &  &-1
\end{pmatrix},\\ 
s_1=
\begin{pmatrix}
0&1&0&0\\
1&0&0&0\\
0&0&0&1\\
0&0&1&0
\end{pmatrix}
s_2=
\begin{pmatrix}
0&0&1&0\\
0&0&0&1\\
1&0&0&0\\
0&1&0&0
\end{pmatrix}.
\end{gather*}
We could already observe that the smaller the dimension and the size
of the base field the harder to find a good pair of vectors. 
Therefore, it is not surprising that the most complicated part is to make 
clear the problem for subgroups of $G\leq GL(4,3)$, $G\leq GL(4,9)$
and $G\leq GL(4,5)$. 
In the first two case we need to use even the assumption
$(|G|,|V|)=1$, while in case of $GL(4,5)$ we found a suitable pair of 
vectors by using a computer. Our next two theorems are about these cases.
\begin{thm}\label{thm:mon_GL(4,3)}
Let $F=A\left<d_1,d_2,s_1,s_2\right>\nor G\leq GL(4,3^k)$.
Furthermore, set $x_1=\ud u_2+\ud u_3+\ud u_4$, and $y_1=\ud u_1$. 
Then $|C_G(x_1)\cap C_G(y_1)|\leq 2$. If the pair of vectors $x_1,y_1$ 
would not be a good choice, then let $1\neq g_0\in C_G(x_1)\cap
C_G(y_1)$. Then $g_0$ is a permutation matrix fixing one of the
elements $\ud u_2, \ud u_3,\ud u_4$. We can assume that 
$g_0(\ud u_2)=\ud u_2$.  
Let us define the vectors $x_2,y_2,x_3,y_3\in V$ as
\begin{gather*}
x_2=\ud u_1+\ud u_2+\ud u_4,\qquad y_2=\ud u_1+\ud u_3;\\
x_3=\ud u_1+\ud u_2-\ud u_4,\qquad y_3=\ud u_1+\ud u_3.
\end{gather*}
Now, either $C_G(x_2)\cap C_G(y_2)=1$, or $C_G(x_3)\cap C_G(y_3)=1$.
\end{thm}
\begin{proof}
We know that $C_G(y_1)$ consists of monomial matrices by part
\textsl{1} of Lemma \ref{thm:monlem}, so any $g\in
C_G(x_1)\cap C_G(y_1)$ acts as a permutation on the set
$\{\ud u_2,\ud u_3,\ud u_4\}$. Since the order of $|G|$ is not
divisible by $3$, we get $C_G(x_1)\cap C_g(y_1)$ is isomorphic to a
$3'$-subgroup of the symmetric group $S_3$, so
$|C_G(x_1)\cap C_G(y_1)|\leq 2$, and the first part of the theorem is proved.

Let us assume that
\[
g_0=
\begin{pmatrix}
1&0&0&0\\
0&1&0&0\\
0&0&0&1\\
0&0&1&0
\end{pmatrix}
\in G.
\]
Now, $C_G(\ud u_1+\ud u_3)$ normalizes the subgroup
$N=C_F(\ud u_1+\ud u_3)$ generated by the elements $d_2,s_2$.
It is easy to check that the $N$-invariant subspaces
\[
\left<\ud u_1+\ud u_3\right>,\ \left<\ud u_1-\ud u_3\right>,\ 
\left<\ud u_2+\ud u_4\right>,\ \left<\ud u_2-\ud u_4\right>
\]
are pairwise non-equivalent representations of $N$. Hence
$C_G(\ud u_1+\ud u_3)$ permutes these subspaces. (In other words, it
consists of momomial matrices with respect to this new basis.) 
Of course, $C_G(\ud u_1+\ud u_3)$ fixes the subspace
$\left<\ud u_1+\ud u_3\right>$. Using again that $|G|$ is not
divisible by $3$, we get at least one of the following holds:
\begin{gather*}
\forall g\in C_G(\ud u_1+\ud u_3)\Rightarrow g(\ud u_1-\ud u_3)=\alpha_g (\ud
u_1-\ud u_3)\textrm{\quad for some }\alpha_g\in\FF {p^a}^*,\\
\forall g\in C_G(\ud u_1+\ud u_3)\Rightarrow g(\ud u_2+\ud u_4)=\alpha_g (\ud
u_2+\ud u_4)\textrm{\quad for some }\alpha_g\in\FF {p^a}^*,\\
\forall g\in C_G(\ud u_1+\ud u_3)\Rightarrow g(\ud u_2-\ud u_4)=\alpha_g (\ud
u_2-\ud u_4)\textrm{\quad for some }\alpha_g\in\FF {p^a}^*.
\end{gather*}
In the first case $C_G(\ud u_1+\ud u_3)$ fixes both the $\left<\ud u_1,\ud
u_3\right>$ and the $\left<\ud u_2,\ud u_4\right>$ subspaces. 
Thus, if a $g\in C_G(\ud u_1+\ud u_3)$ fixes either $x_2$ or $x_3$,
then $g(\ud u_1)=\ud u_1$, and $g$ is a monomial matrix. Furthermore,
either $g=1$, or $g(\ud u_2)=\beta \ud u_4$ and 
$g(\ud u_4)=\gamma \ud u_2$ for some $\beta,\gamma\in \FF {p^a}^*$. 
However, it that case the order of $g_0g\in G$ is divisible
by three, a contradiction. So, in this case we get
$C_G(x_2)\cap C_G(y_2)=C_G(x_3)\cap C_G(y_3)=1$.

In the second case we claim that $C_G(x_2)\cap C_G(y_2)=1$. 
Let $g\in C_G(x_2)\cap C_G(y_2)$. If $g(\ud u_1-\ud u_3)=\beta (\ud
u_1-\ud u_3)$ for some $\beta\in \FF {p^a}^*$, then $g=1$ by the 
previous paragraph.
Otherwise, $g(\ud u_1-\ud u_3)=\gamma(\ud u_2-\ud u_4)$ holds for some
$\gamma\in\FF {p^a}^*$. 
Using that $\frac 12=-1$ in $\FF {3^k}$ we get
\begin{multline*}
g(\ud u_1+\ud u_2+\ud u_4)=\frac{1}{2}(g(\ud u_1+\ud u_3)+g(\ud
u_1-\ud u_3))+g(\ud u_2+\ud u_4)=\\
-(\ud u_1+\ud u_3)+ \gamma(\ud u_2-\ud u_4) 
+\alpha_g (\ud u_2+\ud u_4)\neq \ud u_1+\ud u_2+\ud u_4.
\end{multline*}
This contradiction shows that $C_G(x_2)\cap C_G(y_2)=1$.

Finally, in the third case the proof of $C_G(x_3)\cap C_G(y_3)=1$ 
is essentially the same as the proof was in the second case.
\end{proof}
\noindent\textbf{Remark}

\medskip\noindent
In the above example, if we start from the decomposition $F=AD_0'S'$,
where $D_0'=\left<d_2,s_2\right>$ and $S=\left<d_1,s_1\right>$, 
then the corresponding bases $\{\ud u_1',\ud u_2',\ud u_3',\ud u_4'\}$ 
suitable to Theorem \ref{thm:fitting_mon} will be the following
\[
\ud u_1'=\ud u_1+\ud u_3,\quad \ud u_2'=\ud u_1-\ud u_3,\quad
\ud u_3'=\ud u_2+\ud u_4,\quad \ud u_4'=\ud u_2-\ud u_4.
\]
Written in this new basis, the vectors $x_2,y_2,x_3,y_3$ have the
following form
\begin{gather*}
x_2=-\ud u_1'-\ud u_2'+\ud u_3',\qquad y_2=\ud u_1';\\
x_3=-\ud u_1'-\ud u_2'+\ud u_4',\qquad y_3=\ud u_1'.
\end{gather*}
Hence in case of $G\leq GL(4,3)$ we can assume that there exists a
 pair $x,y$ such that $C_G(x)\cap C_G(y)=1$, where
$y=\ud u_1$, and $x$ is the linear combination of exactly three basis
 vectors with non-zero coefficients.
\bigskip\\
In case of  $GL(4,5)$ we used the GAP system 
\cite{gap} to find suitable $x$ and $y$.
\begin{thm}
As before, let
$F=A\left<d_1,d_2,s_1,s_2\right>\leq GL(4,5)$, and let $N$ denote the
normalizer of $F$ in $GL(4,5)$. Then, for
$x=\ud u_1+\ud u_2+2\ud u_3,\ y=\ud u_2+\ud u_3+2\ud u_4$, we have
$C_N(x)\cap C_N(y)=1$.
\end{thm}
Finally, if the size of the base field is not equal to $3,5$ or $9$, then the
following theorem guarantees the existence of a good pair of $x$ nd $y$.
\begin{thm}\label{thm:mon_GL(4,7)}
As in the previous theorems let $F=A\left<d_1,d_2,s_1,s_2\right>
\nor G\leq GL(4,p^a)$, and assume that $p^a\neq 3,\,5,\,9$. 
Furthermore, let $\alpha\in \FF{p^a}$ be a generator of the
multiplicative group of $\FF{p^a}$. Set
$x=\ud u_2+\alpha \ud u_3+\alpha^{-1}\ud u_4,\ y=\ud u_1$.
Then $C_G(x)\cap C_G(y)=1$. 
\end{thm}
\begin{proof}
Let $g\in C_G(x)\cap C_G(y)$.
By the choice of $y$ we know that $g$ is a monomial matrix. 
The first element in the main diagonal of $\delta(g)$ is 1, and the others
are from the set $\{1,\alpha,\alpha^{-1}\alpha^2,\alpha^{-2}\}$. If
$\delta(g)$ contains an $\alpha$ or an $\alpha^{-1}$, then for some
$s\in S$ we get $[\delta(g),s]\in A\times D_0$ contains both $\alpha$
and $\alpha^{-1}$. By part \textsl{4} of Theorem
\ref{thm:fitting_mon}, this is impossible unless $o(\alpha^2)|4$,
which cannot hold by our assumption to $p^a$. 
It follows that either $g=1$, or
\[
g=
\begin{pmatrix}
1&0&0&0\\
0&1&0&0\\
0&0&0&\alpha^2\\
0&0&\alpha^{-2}&0
\end{pmatrix},\qquad\textrm{and}\qquad
[\delta(g),s_2]=
\begin{pmatrix}
\alpha^2&0&0&0\\
0&\alpha^{-2}&0&0\\
0&0&\alpha^{-2}&0\\
0&0&0&\alpha^2
\end{pmatrix}.
\]
It follows from $[\delta(g),s_2]\in D$ that $o(\alpha^4)|2$, which 
is again impossible, since $p^a\neq 3,\,5,\,9$.
\end{proof}
The constructions given in the last three theorems have the common
property that $x$ is a sum of exactly three basis vectors with
non-zero coefficient. Capitalizing this property, we shall give 
a uniform construction in any case of 
$F=AD_0S\nor G\leq GL(2^k,p^a)$ for all $k\geq 3$.
Possibly taken a permutation of the basis vectors
 $\ud u_1,\ud u_2,\ldots,\ud u_e$ we can assume that
 $\{\ud u_1,\ud u_2,\ud u_3,\ud u_4\}$ corresponds to a two
dimensional subspace of $S$, that is,
\[
\{\ud u_1,\ud u_2,\ud u_3,\ud u_4\}=S_2(\ud u_1)=\{s(\ud u_1)\,|\,s\in
S_2\}\quad \textrm{for some } S_2\leq S,\ |S_2|=4.
\]
Let $V'=\left<\ud u_1,\ud u_2,\ud u_3,\ud u_4\right>\leq
V$ the subspace generated by the first four basis vectors, and
$N_F(V')$ the subgroup of elements of $V'$ fixing $F$.
Now, $N_F(V')/C_F(V')$ is included into $GL(V')$ by restriction to
$V'$, so we get a subgroup $F'=A\left<d_1,d_2,s_1,s_2\right>\leq GL(V')$.
If $g\in N_G(V')$, then it is clear that $g_{V'}$ normalizes
$F'$-t. Using the previous results, we can define $x_0,y_0\in V'$
such that $x_0$ is the linear combination of exactly three basis
vectors and $N_G(V')\cap C_G(x_0)\cap C_G(y_0)$ acts trivially on $V'$.
Starting from the pair $x_0,y_0$, we sarch a good pair of vectors
$x,y\in V$ in the form $x=x_0,\ y=y_0+v$, where
$v\in V'':=\left<\ud u_5,\ud u_6,\ldots,\ud u_e\right>$. 
The following lemma answers the question why this form is good.
\begin{lem} $C_G(x_0)$ fixes both the $V'$ and the $V''$ subspaces,
that is, $C_G(x_0)\leq N_G(V')\cap N_G(V'')$.
As a result, for any $v\in V''$ we have 
$C_G(x_0)\cap C_G(y_0+v)=C_G(x_0)\cap C_G(y_0)\cap C_G(v)$ acts
trivialy on $V'$ In particular, $C_G(x_0)\cap C_G(y_0+v)$ consists of
monomial matrices.
\end{lem}
\begin{proof}
It is enough to prove the inclusion $C_G(x_0)\leq N_G(V')\cap
N_G(V'')$, the rest of the statement follows evidently.
Our proof is similar to how we have proved that $C_G(\ud u_1)$
consists of monomial matrices. As there occurs three basis element in 
$x_0$ and $S\simeq Z_2^{e_1}$ permutes regularly the basis element
we get $C_F(x_0)\leq AD_0$, i.e., every element of $C_F(x_0)$ is diagonal.
Hence every element of $C_F(x_0)$ fixes the three basis element
appearing in $x_0$. Using the assumption that
$\ud u_1,\ud u_2,\ud u_3,\ud u_4$ corresponds to the subspace $S_2\leq S$,
it follows easily that any element of $D_0$ fixing three of the basis
elements $\ud u_1,\ud u_2,\ud u_3,\ud u_4$ must fix the fourth one,
too. From our choice $D_0=C_D(\ud u_1)$
(see proof of Theorem \ref{thm:fitting_mon}) $N:=C_F(x_0)=C_{D_0}(S_2)$. 
It is clear from this that $C_{D_0}(S_2)$ is a two codimensional
subspace of $D_0$-nak, so
$V'$ is just $N$ the homogeneous component of $N$ corresponding to the
trivial representation, while 
$V''$ is the sum of all of the other homogeneous component of $N$. 
(These homogeneous components corresponds to cosets of $S_2$ in
$S$.) As $N\nor C_G(x_0)$, we get every element of $C_G(x_0)$ permutes
the homogneous components of $N$. Since $x_0\in V'$, we get $C_G(x_0)$
fixes $V'$, so it also fixes the sum of the other components, which is $V''$.
\end{proof}
It is time to define the vector $v$, whereby we close the monomial
case. We already know from the previous
lemma that $C_G(x_0)\cap C_G(y_0+v)$ consists of monomial matrices 
for any $v\in V''$, so we can use the constructions given in Theorem
\ref{thm:affin_perm} to define a $\pi(C_G(x_0)\cap
C_G(y_0))$-regular partition on the space
$W=\{\ud u_1,\ud u_2,\ldots,\ud u_e\}$.
\begin{thm}\label{thm:mon_GL(2^k)}
By part  \ref{thm:affin_perm} of Theorem \textsl{5-6} let 
$W=\Omega_1\cup\Omega_2\cup\ldots\cup\Omega_5$ be a $\pi(C_G(x_0)\cap
C_G(y_0))$-regular partition of $W=\{\ud u_1,\ud u_2,\ldots,\ud
u_e\}$ such that $\Omega_1=\{\ud u_1\}\cup\Omega_2\cup\Omega_3\leq 
\{\ud u_1,\ud u_2,\ud u_3,\ud u_4\}$.
(We can achieve this by choosing a suitable $S$.)
Let the vectors $x,y\in V$ be defined as follows
\[
x=x_0,\ y=y_0+v,\textrm{\qquad where\ \ } v=0\cdot\sum_{\ud
u_i\in\Omega_4}\ud u_i+
1\cdot\sum_{\ud u_i\in\Omega_5}\ud u_i, \textrm{\qquad for\ \ }e\neq 16.
\]
In case of $e=16$ this construction is not effective (since it was an 
exceptional case in Corollary \ref{kov:affin_perm}). In this case let
$\ud u_s,\ud u_t\in \{\ud u_5,\ud u_6\ldots,\ud u_{16}\}$ 
be two vectors corresponding to elements from different cosets of 
$S_2$ in $S$.
In this case let $x,y\in V$ be chosen as
\[
x=x_0,\qquad y=y_0+0\cdot \ud u_s+(-1)\cdot \ud u_t+1\cdot
\sum_{\twolineindex{i\in\{5,6,\ldots,16\}}{i\neq s,t}}\ud u_i.
\]
The we have $C_G(x)\cap C_G(y)=1$.
\end{thm}
\begin{proof}
We know by the previous lemma that any $g\in C_G(x)\cap C_G(y)$ is
a monomial matrix fixing all the vectors 
$\ud u_1,\ud u_2,\ud u_3,\ud u_4$, so it fixes all of the sets
$\Omega_1,\Omega_2,\Omega_3$. In case $e\neq 16$ even $\Omega_4$ is
fixed by $\pi(g)$, since exactly the element from $\Omega_4$ are
colored by $0$. It follows that $\pi(g)=1$. Hence $g=\dg$ is a
diagonal matrix, and any element in its main diagonal not
corresponding to $\Omega_4$ must be $1$. However,
$|\Omega_4|<1/4|W|$ by Corollary \ref{kov:affin_perm}, so we get
$g=\dg=1$ by part \textsl{4} of Lemma
\ref{thm:monlem}.

In case $e=16$ for the permutation part of any element 
$g\in C_G(x)\cap C_G(y)$ we have
$\pi(g)(\ud u_s)=\ud u_s$. Now, if $\pi(g)(\ud u_t)\neq\ud
u_t$ does not hold, then the number of elements in the diagonal of
$\dg$ different from $1$ should be $2$ or $3$, which is again
a contradiction to part \textsl{4} of Lemma  
\ref{thm:monlem}.
Hence $\dg=1$ and $\pi(g)(\ud u_t)=\ud u_t$. By choice of the vectors 
$\ud u_s,\ud u_t$ we get $g=\pi(g)=1$, which proves the identity 
$C_G(x)\cap C_G(y)=1$.
\end{proof}
\subsection{Finding $x,y\in V$ in case $F$ is not monomial}
Now, we handle the case when $F$ is not monomial. Thus, the extraspecial
2-group, say $P_1$, in the decomposition of $F\nor G\leq
GL(V)\simeq GL(e,p^a)$ corresponding to part \textsl{4} of Theorem
 \ref{thm:maxprim} is the central product of a quaternion group $Q$
by some (maybe $0$) dieder groups $D_4$. 
If $\lambda\in A$ is a field element of order four, and 
$Q=\left<i,j\right>\leq P_1$ is the quaternion group geerated by the
element $i,j$ of order four, then defining
$H=\left<\lambda i,\lambda j\right>\leq AQ$ we get $H\simeq D_4$
and $AH=AQ$. These means that in the decomposition of $F$ 
we can exchange $Q$ for a subgroup isomorphic to $D_4$, so we get the
monomial case. Therefore, we can assume that $A$ does not contain a
fourth root of unity.
Our next theorem is analogous to Theorem \ref{thm:fitting_mon}. 
\begin{thm}\label{thm:fitting_nonmon}
  With the above assumptions, the subgroup $F\leq GL(V)$ has the
  following properties
  \begin{enumerate}
  \item There exists a (not necessirily direct) 
   product decomposition $F=QF_1$ such that 
   $F_1=C_F(Q)=D\rtimes S=(A\times D_0)\rtimes S$ 
   and
    \[
	D_0\simeq S\simeq Z_{2}^{e_1-1}\times
	Z_{p_2}^{e_2}\times\ldots\times Z_{p_k}^{e_k}.
    \]
  \item There is a basis $\ud u_1,\ud v_1,\ud u_2,\ud v_2,
    \ldots, \ud u_{e/2},\ud v_{e/2}\in V$ such that written in this basis
    the elements of $D$ are diagonal matrices, while $S$ permutes the set
    of ordered pairs $\{(\ud u_i,\ud v_i)\,|\,1\leq i\leq e/2\}$ regularly.
  \item The subspaces $\left<\ud u_i\right>$ are all the irreducible
  representations of $D_0$ over $\FF {p^a}$ and they are pairwise
  non-equivalent.
  \item For any $g\in D_0$, the main diagonal of $g$ contains all of
  the $o(g)$-th root of unity with the same multiplicity.
  \item For all $1\leq i\leq e/2$ any element of $D$ restricting to
    $W_i=\left<\ud u_i,\ud v_i\right>$ is a scalar matrix. 
  \item If an element $g\in QD$ has an eigenvalue (in this
  representation), then $g\in D$.
  \end{enumerate}
\end{thm}
\begin{proof}
If $P_1=QT$ is the central product of the quaternion group $Q$ 
and the extraspecial $2$-group $T$ (which is itself a central product
of some $D_4$'-s), then we can apply Theorem \ref{thm:fitting_mon} to
the group $F_1=ATP_2P_3\ldots P_k$. Hence the first statement follows
at once from part \textsl{1} of Theorem \ref{thm:fitting_mon}.

Let $V_1\leq V$ be an irreducible $F_1$-invariant subspace of $V$. 
By Theorem \ref{thm:fitting_mon} the dimension of
$V_1$ is $e/2$, firthermore, there exists a basis 
$\{\ud u_1,\ud u_2,\ldots \ud u_{e/2}\}\in V_1$ of
$V_1$ such that $D$ consists of diagonal matrices respect to this
basis, while $S$ permutes regularly the elements of this basis. 
Now, statement \textsl{3} is just the redefinition of the
corresponding part of Theorem \ref{thm:fitting_mon}.

Let $W_i=\left<q(\ud u_i)\,|\,q\in Q\right>$ be the smallest
$Q$-invariant subspace containing $\ud u_i$. 
Then each $W_i$ is a homogeneous $D_0$-module, 
since $Q$ centralizes $D_0$, so \textsl{5}
follows. Additionally, these subspaces are pairwise non-equivalent 
$D_0$-modules. 

Since $Q$ centralizes also $S$, we get $S$ permutes regularly the
subspaces $W_i$.
It follows that $W_1\oplus W_2\oplus\ldots\oplus W_{e/2}$ is an
$F$-invariant subspace, so it is equal to $V$ by part \textsl{7} of
Theorem \ref{thm:maxprim}. Comparing dimensions we get each $W_i$ is
two dimensional.
Let us choose elements $\ud v_i\in W_i$ such that $\ud u_i,\ud v_i$ is
a basis of $W_i$ for all $i$, and
the set of vectors $\{\ud v_1,\ud v_2,\ldots,\ud v_{e/2}\}$
is an orbit of $S$. Now, \textsl{2} follows obviously.

Using the corresponding part of the monomial case it
follows \textsl{4} at once.

Finally, let $g=qd\in QD\setminus D$, so $q\in Q\setminus\{\pm I\}$.
As the elements of $Q$ are commutable with the elements of 
$D$ and the exponent of $D$ is not divisible by $4$
(Here we use that $A$ does not contain a fourth root of unity), 
we get the order of $g$ is divisible by four. 
It follows that $g^{o(g)/2}$ is an element of $Q$ of order two, hence 
$g^{o(g)/2}=-I$. Now, if $\lambda$ is an eigenvalue of
$g$, then $\lambda^{o(g)/4}\in \FF {p^a}$ would be a fourth root of unity,
a contradiction. Hence any element of $QD\setminus D$ does not have an
eigenvalue, which proves \textsl{6}.
\end{proof}
\noindent According to the last theorem let $V_1=\left<\ud u_1,\ud
u_2,\ldots,\ud u_e\right>$ and $V_2=\left<\ud v_1,\ud
v_2,\ldots,\ud v_e\right>$. Then $V=V_1\oplus V_2$.
Let $N_G(V_1)$ denote the elements of $G$ fixing the subspace
$V_1$. Then the restriction of $G_1=N_G(V_1)/C_G(V_1)$ to $V_1$
gives us an inclusion $G_1\leq GL(V_1)$. It is clear that $G_1$ 
contains the restriction of $F_1$ to $V_1$ as a normal subgroup.
Using the constructions of the monomial case,
we can find vectors $x_1,y_1\in V_1$ such that 
$C_{G_1}(x_1)\cap C_{G_1}(y_1)=1_{V_1}$.
Furthermore, in cases $e/2\neq 2^t$ and $e/2=2$ we have $x_1=\ud u_1$
by Theorem \ref{thm:mon_eneq2^k}, while in cases
$e/2=2^t,\ t\geq 2$ we found $x_1\in\left<\ud u_1,\ud u_2,\ud u_3,\ud
u_4\right>$ as a linear combination of exactly three basis vectors, 
while $y\in \ud u_1+\langle\ud u_5,\ud u_6,\ldots,\ud u_{e/2}\rangle$.
(Theorems \ref{thm:mon_GL(4,3)}-\ref{thm:mon_GL(4,7)}, \ref{thm:mon_GL(2^k)}, 
 and Remark after Theorem \ref{thm:mon_GL(4,3)})
Starting from these constructions we define vectors $x,y\in V$ as follows.
\begin{thm}
Using the vectors $x_1,y_1\in V_1$ defined above let
\[
\begin{array}{lll}
x=x_1,	      &y=\ud v_1+y_1,&\textrm{ in cases $e/2\neq 2^k$ or $e/2=2$;}\\ 
x=\ud v_1+x_1,&y=y_1,        &\textrm{ in cases $e/2=2^k,\ k\geq 2$.}
\end{array}
\]
Then $C_G(x)\cap C_G(y)=1$.
\end{thm}
\begin{proof}
First, let $e/2\neq 2^k$ or $e/2=2$.  
Choosing a $g\in C_G(x)\cap C_G(y)$ it normalizes the subgroup
$C_F(x)=C_F(\ud u_1)=D_0$, so it permutes the homogeneous components
of $D_0$, that is, the subspaces $W_1,W_2,\ldots,W_{e/2}$.
Then it is clear from the construction of $y$ that $g$ also
centralizes $\ud v_1$, so the restriction of $g$ to
$W_1$ is the identity.
As $g$ permutes the subspaces $W_1,W_2,\ldots,W_{e/2}$
it follows that $g$ can be written in a unique way as a product 
$g=\delta_2(g)\pi_2(g)$, where $\delta_2(g)$ is a $2$-block diagonal
matrix, while $\pi_2(g)=\pi(g)\otimes I_2$, where $\pi(g)$ denotes 
the permutation action of $g$ on the set
$\{W_1,W_2,\ldots,W_{e/2}\}$. Similarly to part \textsl{3} of Lemma
(\ref{thm:monlem} one can prove that $\delta_2(g)$ must normalize $F$,
as well.
Now, if $\ud u_i$ appears with a non-zero coefficient in $y$, then the
$i$-th block of $\delta_2(g)$ must be a upper triangular matrix. 
If we choose $s\in S$ such that $s(\ud u_1)=\ud u_i$, then the first
block of the 2-block diagonal matrix
$[\delta_2(g),s]\in QD$ is the same as the $i$-th block of
$\delta_2(g)$. As a upper triangular matrix does have an eigenvalue, 
by part \textsl{6} of Theorem \ref{thm:fitting_nonmon} we get
$[\delta_2(g),s]\in D$, so every block of $[\delta_2(g),s]$, 
in particular, the first one, is scalar matrix. 
Thus, we showed that for any $\ud u_i$ appearing in $y$ the
corresponding block of $\delta_2(g)$ is scalar matrix.
Such $\ud u_i$'-s are in bijection with the elements of some $\Omega_i$'-s
in Theorems \ref{thm:affin_perm} and \ref{thm:kevert}. (see also
Theorem \ref{thm:mon_eneq2^k}
It is easy to check that in any such case more than half of the $\ud u_i$'-s 
appears in $y$, so more than half of the blocks of $\delta_2(g)$ is
scalar matrix. It follows that for any $s\in S$ at least one block of 
$[\delta_2(g),s]\in QD$ is a scalar matrix. Using part  \textsl{6} of
Theorem \ref{thm:fitting_nonmon} again, we get $[\delta_2(g),s]$ is
diagonal matrix for all $s\in S$. Since the first block of
$\delta_2(g)$ is the identity, and $S$ regularly permutes the blocks
we get every block of $\delta_2(g)$ is scalar matrix, that is,
$\delta_2(g)$ is diagonal. Hence $g$ is monomial, and it fixes the 
subspace $V_1=\langle\ud u_1,\ud u_2,\ldots,\ud u_{e/2}\rangle$. 
As $g_{V_1}\in C_{G_1}(x_1)\cap C_{G_1}(y_1)=1_{V_1}$, we have $g$ 
acts on $V_1$ trivially, so $\pi(g)=1$, and $g$ is a diagonal matrix. 
Finally, using that the restriction of $g$ to any $W_i$ is a scalar
matrix, and $g(\ud u_i)=\ud u_i$ for all $i$ it follows that $g=1$,
what we wanted to prove.

In case $e=2^k,\ k\geq 2$ we claim that
$C_F(x)=C_F(\ud v_1)\cap C_F(x_1)\leq D_0$. 
As the set of subspaces $W_1,W_2,W_3,W_4$ corresponds to a subspace of
$S$, it follows that $C_F(x)$ permutes these subspaces.
Now, if  $g\in C_F(x)$ takes $\ud u_i$ into a multiple of $\ud u_j$
for some $\ud u_i,\ud u_j$ occuring in $x$, then $\ud u_j$ is an
eigenvalue of the 2-block diagonal part $\delta_2(g)\in QD$ of $g$,
hence $\delta_2(g)$ is diagonal by part \textsl{6} of
Theorem \ref{thm:fitting_nonmon}. Consequently, $g$ cannot take $\ud
v_1$ into a multiple of some $\ud u_i$. So $C_f(x)$ fixes both $\ud
v_1$ and $x_1$, which proves that $C_F(x)=C_F(\ud v_1)\cap C_F(x_1)\leq D_0$.

It follows that the homogeneous component corresponding to the trivial
representation of $C_F(x)\leq D_0$ is just the subspace 
$W_1\oplus W_2\oplus W_3\oplus W_4$, while the subspace generated by
the other homogeneous components of $C_F(x)$ is 
$W_5\oplus W_6\oplus\cdots\oplus W_{e/2}$.  
Since any $g\in C_G(x)\cap C_G(y)$ normalizes $C_F(x)$,
it permutes these homogeneous components. We get $g$ fixes
both $W_1\oplus W_2\oplus W_3\oplus W_4$ and
$W_5\oplus W_6\oplus\cdots\oplus W_{e/2}$. As $y$ is of the form
$y=\ud u_1+\langle\ud u_5,\ud u_6,\ldots,\ud u_{e/2}\rangle$,
it follows that $g(\ud u_1)=\ud u_1$, so $g$ fixes the subspace $W_1$,
and permutes the subspaces $W_2,\ldots,W_{e/2}$. 
Using the construction of $x$ we get $g(\ud v_1)=\ud
v_1$, so $g$ acts trivially on $W_1$. 
From this point our proof is the same as it was for the previous case.
\end{proof}
\section{Imprimitve linear groups}\label{2bas:imprim}
As before, let $p\neq 2$ prime (or prime power), let $V$ be a finite
vector space over $\FF p$ and $G\leq GL(V)\simeq GL(n,p)$ 
solvable linear group such that $(|G|,|V|)=1$. 
In case of $G$ is a primitive linear group, the previous section gave
us a base $x,y\in V$.
Using this result, in this section we handle the case, when $G$ is not
primitive as a linear group.\par 
It follows from Maschke'-s theorem that $V$ is an completely reducible
$G$-module. The next obvious lemma reduce the problem to irreducible 
$G$-modules.
\begin{lem}
Let $V=V_1\oplus V_2$ the sum of two $G$-invariant subspaces. Now,
$G/C_G(V_i)\leq GL(V_i)$ acts faithfully on $V_i$. For $i=1,2$, set 
$x_i,y_i\in V_i$ such that $C_G(x_i)\cap C_G(y_i)=C_G(V_i)$.
Then $C_G(x_1+x_2)\cap C_G(y_1+y_2)=1$.
\end{lem}
Let $G\leq GL(V)$ be an irreducible, imprimitive linear group.
Thus, there is a decomposition $V=\oplus_{i=1}^k V_i$
such that $k\geq 2$ and $G$ permutes the subspaces $V_i$ in a
transitive way. We can assume that the decomposition cannot be refined.
For each $1\leq i\leq k$ let $H_i=\{g\in
G\;|\; gV_i=V_i\}$ be the stabilizer of $V_i$ in $G$. 
Then $H_i/C_{H_i}(V_i)\leq GL(V_i)$ is a linear group, and the 
subgroups $H_i$ are conjugate in $G$. 
Of course, $(|H_1|,|V_1|)=1$, so, using the previous section
we can find vectors $x_1,y_1\in V_1$ such that 
$C_{H_1}(x_1)\cap C_{H_1}(y_1)=C_{H_1}(V_1)$.
Let $\{g_1=1, g_2,\ldots,g_k\}$ be a set of right coset
representatives to $H_1$ in $G$ such that 
$V_i=g_iV_1$ for all $1\leq i\leq k$-ra, and let $x_i=g_ix_1$,
 $y_i=g_iy_1$. It is clear that
$H_i=H_1^{g_i^{-1}}$ and $C_{H_i}(x_i)\cap C_{H_i}(y_i)=C_{H_i}(V_i)$.\par
Now, $N=\cap_{i=1}^k H_i$ is a normal subgroup of $G$, the quotient group 
$G/N$ acts faithfully and transitively on the set 
$\{V_1,V_2,\ldots,V_k\}$, and $|G/N|$ is coprime to $p$. 
Using Theorem \ref{thm:perm}, we can choose a vector
$(a_1,a_2,\ldots a_k)\in \FF p^k$ such that (to the above permutation action)
only the identity element of $G/N$ fixes this vector.
\begin{thm}
Let the vectors $x,y\in V$ be defined as
\[
x=\sum_{i=1}^k x_i,\qquad\qquad y=\sum_{i=1}^k (y_i+a_ix_i).
\]
Then $C_G(x)\cap C_G(y)=1$.
\end{thm}
\begin{proof}
Let $g\in C_G(x)\cap C_G(y)$. Assuming that $gV_i=V_j$ for some
$1\leq i,j\leq k$ we get $gx_i=x_j$ and $g(y_i+a_ix_i)=(y_j+a_jx_j)$.
Choose $g'=g_j^{-1}gg_i\in G$. Ekkor
\begin{equation}\label{eq:imprim:unitri}
g'x_1=x_1 \textrm{\quad and\quad} g'(y_1+a_ix_1)=(y_1+a_ix_1)+(a_j-a_i)x_1,
\end{equation}
so $g'$ stabilizes the subspace $\left<x_1,y_1\right>\leq V_1$. 
If $y_1=cx_1$ for some $c\in \FF p$, then $g'y_1=y_1$. 
Using the identity (\ref{eq:imprim:unitri}) we get $a_j=a_i$. Otherwise,
$x_1,y_1+a_ix_1$ form a basis of the subspace $\langle x_1,y_1\rangle$
which is a two dimensional $g'$-invariant subspace. 
With respect to the basis $x_1,y_1+a_ix_1$, the restriction of 
$g'$ to this subspace has matrix form 
\[
\begin{pmatrix}
1&a_j-a_i\\
0&1\\
\end{pmatrix}.
\]
If $a_j-a_i\neq 0$, then this matrix has order $p$, so $p$ divides the
order of $g'\in G$, a contradiction. Hence in any case 
$a_i=a_j$ holds for $gV_i=V_j$, which exactly means that $gN\in G/N$ 
stabilizes the vector $(a_1,a_2,\ldots,a_k)$. 
It follows that $g\in N$. So $gx_i=x_i$ and $gy_i=y_i$ holds for any 
$1\leq i\leq k$, and $g\in\cap_{i=1}^k C_{H_i}(V_i)=C_G(V)=1$ follows.
\end{proof}

\bigskip\bigskip
\begin{tabular}{l@{\qquad\quad}l}
Zolt\'an Halasi                  &K\'aroly Podoski\\ 
Central European University      &Alfr\'ed R\'enyi Institute of Mathematics\\ 
Department of Mathematics        &Hungarian Academy of Sciences\\
 and Its Applications		 &\\ 
H-1051 Budapest			 &H-1364 Budapest\\
N\'ador utca 9.			 &P. O. Box 127\\
Hungary				 &Hungary\\ 
e-mail: haca@cs.elte.hu		 &e-mail: pcharles@cs.elte.hu
\end{tabular}
\end{document}